\documentclass[11pt]{article}

\usepackage{amssymb,amsfonts,amsmath}
\usepackage{amsthm}
\usepackage{algorithmic}
\usepackage{algorithm}
\usepackage{array}
\usepackage{booktabs}
\usepackage{caption} 
\usepackage{color}
\usepackage{empheq}
\usepackage{graphicx}
\usepackage{placeins}
\usepackage{lipsum}
\usepackage{mathrsfs}
\usepackage{multirow}
\usepackage{natbib}
\usepackage{textcomp}
\usepackage{stfloats}
\usepackage{subfigure}
\usepackage{url}
\usepackage{verbatim}
\usepackage[active]{srcltx}
\usepackage{tabularx}
\usepackage{setspace}
\usepackage{rotating}

\usepackage[margin=2.5cm,nohead]{geometry}
\theoremstyle{plain}
\newtheorem{theorem}{Theorem}
\newtheorem{lemma}{Lemma}
\newtheorem{corollary}{Corollary}
\newtheorem{proposition}{Proposition}
\theoremstyle{definition}

\newtheorem{example}{Example}
\newtheorem{assumption}{Assumption}
\theoremstyle{remark}
\newtheorem{remark}{Remark}

\DeclareMathOperator{\dist}{dist}
\DeclareMathOperator{\sign}{sign}
\newcommand{\alg}[1]{\textsc{#1}}

\graphicspath{{fig/}}

\begin{document}

\title{Error Bounds and a Smoothing Penalty Method for Binary Orthogonal Optimization}

\author{Ruikai Lv, Lianghai Xiao\thanks{Corresponding author. Email: \texttt{xiaolh@jnu.edu.cn}.}}

\date{}
\maketitle
\begin{abstract}
Binary orthogonal optimization arises in semantic hashing, orthogonal binary
projection, and Hadamard matrix design.  Existing methods 
often handle binarity and orthogonality separately, so their outputs
need not satisfy both constraints simultaneously and generally come without
a residual-based guarantee of exact feasibility.  Using an
exact box-constrained reformulation, we establish local and global Lipschitz
error bounds and prove that the Lipschitz exponent is optimal.  In particular,
when the binary orthogonal feasible set is nonempty, any box-constrained matrix
whose orthogonality residual is below $1$ admits a feasible sign rounding, and
the rounding distance is bounded by this residual.  These bounds also yield
a globally exact nonsmooth orthogonality-residual penalty over the box and a
uniformly sharp local penalty threshold.  We further develop a smoothing 
penalty method for binary orthogonal optimization (SPBO). 
We prove that every accumulation point is stationary for the limiting 
nonsmooth penalty problem. Numerical experiments on
orthogonal binary projection, supervised hashing, and Hadamard completion show
that SPBO substantially reduces orthogonality violations while retaining
competitive objective values or retrieval accuracy.  
\end{abstract}
\noindent


%
\section{Introduction}
Binary representations reduce storage and permit efficient bitwise
computation in image retrieval, signal processing, and structured
matrix design.  In many of these applications, the binary variables are also
required to be orthogonal, which removes redundancy between columns and
improves the conditioning of subsequent retrieval or recovery procedures.
Given integers $1\le r\le n$ and a continuously differentiable, possibly
nonconvex function $f:\mathbb R^{n\times r}\to\mathbb R$, we consider
\begin{equation}\label{eq:BO}
 \min_{X\in\mathbb R^{n\times r}} f(X)
 \qquad\mathrm{s.t.}\qquad
 X\in\mathcal F_{\rm BO}:=
 \{X\in\{-1,1\}^{n\times r}:X^\top X=nI_r\}.
\end{equation}
The simultaneous presence of binary and orthogonality constraints makes
\eqref{eq:BO} a combinatorial optimization problem. The following examples
illustrate three representative instances.

\begin{example}[Hash learning and image retrieval]
In binary hashing, a code matrix is learned by matching semantic or pairwise
similarities between samples.  A representative model is
\[
 \min_{X\in\mathcal F_{\rm BO}}\ell(X;\mathcal S),
\]
where $\mathcal S$ is constructed from class labels or other similarity
information and $\ell$ is a task-dependent loss
\cite{weiss2008spectral,liu2014discrete,jin2021unsupervised,
sun2023hierarchical}.  Binary feasibility gives compact codes, while
orthogonality promotes decorrelated hash bits.
\end{example}

\begin{example}[Hadamard matrix construction with prescribed entries]
When $r=n$, $\mathcal F_{\rm BO}$ is the set of Hadamard matrices \cite{hedayat1999orthogonal}
of order $n$.  Given observed entries $\widetilde X_{ij}$ on an index set
$\mathcal I$, a completion model is
\[
 \min_{X\in\mathcal F_{\rm BO}}
 \frac12\|\mathcal P_{\mathcal I}(X-\widetilde X)\|_F^2,
\]
where $\mathcal P_{\mathcal I}$ retains the entries indexed by $\mathcal I$
and sets all remaining entries to zero.
\end{example}

\begin{example}[Derandomized measurement design in compressed sensing]
Structured sensing matrices, such as subsampled Hadamard matrices, reduce
storage and random-bit requirements and permit fast matrix--vector products
\cite{li2015state,salari2018tdoa}.  For a prescribed row selector $P$, a
task-driven design model is
\[
 \min_{X\in\mathcal F_{\rm BO}}
 \ell\!\left(\frac1{\sqrt n}PX;\mathcal D\right),
\]
where $\mathcal D$ is training or calibration data and $\ell$ measures
sensing or reconstruction performance.
\end{example}

\subsection{Related work}

Existing methods for binary orthogonal optimization typically handle 
binarity and orthogonality in one of two ways: by relaxing the binary 
constraint or by decoupling the two requirements through auxiliary 
variables. Spectral hashing methods relax the binary constraint and 
obtain binary codes by thresholding the resulting continuous solutions 
\cite{weiss2008spectral}. Other methods retain binary variables but 
employ alternating updates, auxiliary variables, or penalty terms to make 
the resulting subproblems tractable \cite{liu2014discrete, 
gong2012iterative, shen2015supervised, xiong2021generalized}.  
These approaches are often inexpensive, but a final
rounding or block update may destroy orthogonality, and the returned matrix
is not guaranteed to satisfy both constraints simultaneously.

Guaranteeing exact feasibility from an approximate continuous solution
requires quantitative control of its distance to the binary orthogonal
feasible set.  Error-bound theory addresses this geometric question.  
Jiang et al.~\cite{jiang2023exact} and Qian et al.~\cite{qian2024error} 
established error bounds for nonnegative orthogonality constraints under 
different structural conditions. Chen et al.~\cite{chen2024tight} extended 
this analysis to sign-constrained Stiefel manifolds and showed that the 
error-bound exponent \(1/2\) is sharp in general when \(1<r<n\).

Xiao et al.~\cite{xiao2026relaxation} reformulated binary orthogonal optimization
over a restricted Stiefel manifold, enforcing orthogonality exactly while
penalizing box-constraint violations.  They established global exactness of
their penalty models.  Their method requires a manifold retraction at each
iteration, and their analysis does not bound the distance from a box-constrained
point to $\mathcal F_{\rm BO}$ in terms of its orthogonality residual.

We instead impose the box constraint exactly and move the orthogonality
residual into the objective.  The error bounds developed below provide both
an exactness threshold for this penalty and distance and sign-rounding
guarantees for approximate solutions.

\subsection{Contributions} 
The main contributions are twofold.

\begin{enumerate}
\item[i] We establish local and global Lipschitz error bounds for the exact
box-constrained reformulation and prove that the Lipschitz exponent is
optimal.  The asymptotic local error-bound constant $1/2$ is shown to be
uniformly sharp.  These bounds yield explicit sign-rounding certificates.
In particular, if $\mathcal F_{\rm BO}\ne\emptyset$, then every
$X\in[-1,1]^{n\times r}$ satisfying $\|X^\top X-nI_r\|_F<1$ obeys
\[
 \sign(X)\in\mathcal F_{\rm BO},
 \qquad
 \dist(X,\mathcal F_{\rm BO})
 \le\|X-\sign(X)\|_F
 \le\|X^\top X-nI_r\|_F.
\]
The error bounds also yield a globally exact nonsmooth
orthogonality-residual penalty with an explicit finite penalty
threshold, as well as a uniformly sharp local threshold.

\item[ii] We develop SPBO, a smoothing penalty method that requires only
box projections.  It combines bounded penalty continuation with
Huber--Moreau smoothing of the nonsmooth orthogonality-residual penalty.
We quantify the uniform smoothing error, establish
$O(\varepsilon^{-2})$ first-order complexity for each fixed smooth
subproblem, and prove stationarity of accumulation points for the limiting
nonsmooth penalty model under increasingly accurate subproblem solves.
Once the residual certificate is satisfied, sign rounding produces an
exactly feasible binary orthogonal matrix, whose objective value can be
further improved by a feasibility-preserving balanced-switch local
refinement.
\end{enumerate}

Numerically, SPBO markedly reduces orthogonality violations of rounded binary
outputs.  It achieves the lowest violation in all 15 projection configurations
and all 20 supervised-hashing configurations, improving on the next-best method
by factors of $3.24$--$1.04\times10^2$ and $33.7$--$885$, respectively.  In
Hadamard completion, it reaches the reported numerical-feasibility success
levels at lower observation ratios than the competing methods.  

\subsection{Organization}

Section~2 presents the basic notation, the exact continuous
reformulation and properties of its feasible set.  Section~3 develops the error bounds,
rounding guarantees, and exact-penalty consequences.  Section~4 presents the
smoothed penalty model, the projected-gradient continuation method, and its
convergence properties.  Section~5
reports the numerical experiments, and Section~6 concludes the paper.

\section{Preliminaries}
\label{sec:preliminaries}

Throughout the paper, the space $\mathbb R^{n\times r}$ is equipped with the trace inner product
$\langle X,Y\rangle=\operatorname{tr}(X^\top Y)$ and the Frobenius norm
$\|\cdot\|_F$.  We use $\|\cdot\|_1$ and $\|\cdot\|_\infty$ for the
vector $\ell_1$ and maximum norms, respectively,
$\|\cdot\|_2$ for the Euclidean norm of vectors and the spectral norm of
matrices, and
$\|X\|_{\max}:=\max_{i,j}|X_{ij}|$ for the entrywise maximum norm.
The space of $r\times r$ symmetric matrices is denoted by $\mathbb S^r$,
and $I_r$ is the identity matrix.  The $j$th column of $X$ is
$X_{\cdot j}$, and $X\circ Y$ denotes the Hadamard product.  For
$t\in\mathbb R$, define $\sign(t)=1$ if $t>0$, $\sign(t)=0$ if $t=0$,
and $\sign(t)=-1$ if $t<0$.  For a matrix $X$, $\sign(X)$ is defined
entrywise.

For a nonempty closed set $\mathcal C$, let
\[
 \dist(X,\mathcal C):=\inf_{Y\in\mathcal C}\|X-Y\|_F,
 \qquad
 \Pi_{\mathcal C}(X):=
 \arg\min_{Y\in\mathcal C}\|X-Y\|_F.
\]
When $\mathcal C$ is closed and convex, the projection is single-valued and
nonexpansive. 

Let $\Omega:=[-1,1]^{n\times r}$ and
$\mathcal B:=\{-1,1\}^{n\times r}$.  The binary orthogonal feasible set
$\mathcal F_{\rm BO}$ is defined in \eqref{eq:BO}.

Define the smooth constraint mapping
\[
 C:\mathbb R^{n\times r}\to\mathbb S^r,
 \qquad C(X):=X^\top X-nI_r.
\]
Its derivative and adjoint with respect to the Frobenius inner products are
\[
 DC(X)[D]=X^\top D+D^\top X,\qquad
 DC(X)^*[\Lambda]=2X\Lambda,
 \quad\Lambda\in\mathbb S^r.
\]
We measure constraint violation by the residual
\[
 F:\mathbb R^{n\times r}\to\mathbb R_+,\qquad F(X):=\|C(X)\|_F.
\]
Unlike $C$, the residual $F$ need not be differentiable at points where
$C(X)=0$.
The following observation, also used in \cite{xiao2026relaxation}, establishes
the exactness of the box reformulation.

\begin{lemma} 
\label{lem:box-binary-equivalence}
\[
\{X\in\Omega:C(X)=0\}=\mathcal F_{\rm BO}. 
\]
\end{lemma}

\begin{proof}
If $X\in\Omega$ and $C(X)=0$, then the $j$th diagonal equation gives
\[
 \|X_{\cdot j}\|_2^2=\sum_{i=1}^n X_{ij}^2=n,
 \qquad j=1,\ldots,r.
\]
Since $X\in\Omega$, each term satisfies $X_{ij}^2\le1$.  Equality of the
sum to $n$ therefore implies $X_{ij}^2=1$ for every $i,j$, and hence
$X\in\mathcal B$.  The reverse inclusion is immediate.
\end{proof}

Thus, \eqref{eq:BO} has the equivalent box-constrained formulation
\begin{equation}\label{eq:box-orth}
 \min_{X\in\Omega}f(X)
 \qquad\mathrm{s.t.}\qquad C(X)=0.
\end{equation}

Projection onto $\Omega$ is componentwise clipping.  For $X\in\Omega$, its normal cone is
\[
 N_\Omega(X):=
 \{G:\langle G,Y-X\rangle\le0\ \ \forall Y\in\Omega\},
\]
i.e., $G_{ij}\ge0$ at $X_{ij}=1$, $G_{ij}\le0$ at $X_{ij}=-1$,
and $G_{ij}=0$ in the interior.  At a vertex $X\in\mathcal B$, the tangent
cone is
\[
 T_\Omega(X)=\{H:X\circ H\le0\}.
\]

Recent modified alternating-projection methods achieve local quadratic
convergence for general, possibly nonconvex feasibility problems under a
constraint nondegeneracy condition \cite{xiao2025quadraticprojection}.
That condition fails for \eqref{eq:box-orth}: at every
$X\in\mathcal F_{\rm BO}$, the lineality space
$T_\Omega(X)\cap[-T_\Omega(X)]$ is $\{0\}$, so its image under $DC(X)$
cannot equal $\mathbb S^r$.  Thus the cited quadratic-convergence theory
does not apply to this formulation.

Since $\mathcal F_{\rm BO}$ is finite, every feasible point is isolated
and hence locally optimal in the usual constrained sense.  The next
observation shows that every feasible point also satisfies the KKT
conditions of the box reformulation.

\begin{proposition}
\label{prop:kkt_degeneracy}
For every $X\in\mathcal F_{\rm BO}$, there exists
$\Lambda\in\mathbb S^r$ such that
\[
 X_{ij}\bigl[\nabla f(X)+2X\Lambda\bigr]_{ij}<0
 \qquad \forall i,j.
\]
Consequently, every feasible point of \eqref{eq:box-orth} satisfies its
KKT conditions.
\end{proposition}

\begin{proof}
Fix $X\in\mathcal F_{\rm BO}$ and choose $\Lambda=-\tau I_r$, where
$\tau>\frac12\max_{i,j}|\nabla f(X)_{ij}|$.
Since $X_{ij}^2=1$,
\[
 X_{ij}\bigl[\nabla f(X)+2X\Lambda\bigr]_{ij}
 =X_{ij}\nabla f(X)_{ij}-2\tau<0
 \qquad \forall i,j.
\]
These inequalities imply
$0\in\nabla f(X)+2X\Lambda+N_\Omega(X)$, which is the KKT stationarity
condition for \eqref{eq:box-orth}.
\end{proof}

Thus, KKT stationarity for \eqref{eq:box-orth} cannot distinguish globally
optimal solutions from other feasible points.  The exact-penalty analysis
below relates global and local minimizers of the discrete problem to
those of a penalty model over the continuous box.

We next record the divisibility conditions on $n$ that are required for
$\mathcal F_{\rm BO}$ to be nonempty.  To express the second condition, recall
that, for a two-level array $D\in\{-1,1\}^{n\times s}$ with $s\ge2$ and a nonempty
$S\subseteq\{1,\ldots,s\}$, its $J$-characteristic is
\cite{tang2001jcharacteristics}
\[
 J_S(D):=\sum_{i=1}^n\prod_{j\in S}D_{ij}.
\]
The array is an $\operatorname{OA}(n,s,2,2)$ if, in every pair of columns,
each of the four sign patterns $(1,1)$, $(1,-1)$, $(-1,1)$, and $(-1,-1)$
occurs exactly $n/4$ times \cite{hedayat1999orthogonal}.  Equivalently, its
$J$-characteristics of orders one and two vanish
\cite{tang2001jcharacteristics,beder2014wordlength}:
\begin{equation}\label{eq:oa_j_characterization}
 J_S(D)=0\qquad \forall S\text{ with }1\le |S|\le2.
\end{equation}

\begin{proposition}
\label{prop:bo_feasibility_conditions}
If $\mathcal F_{\rm BO}\ne\emptyset$, then
\[
 r\ge2\ \to\ 2\mid n,
 \qquad
 r\ge3\ \to\ 4\mid n.
\]
\end{proposition}

\begin{proof}
Fix $X\in\mathcal F_{\rm BO}$.  If $r\ge2$, choose two distinct columns $X_{\cdot j}$ and
$X_{\cdot\ell}$.  Their orthogonality gives
\(
 0=X_{\cdot j}^\top X_{\cdot\ell}
 =\sum_{i=1}^n X_{ij}X_{i\ell}.
\)
Since $X_{ij}X_{i\ell}\in\{-1,1\}$ for $i=1,\ldots,n$, the sum can be zero
only if it contains equally many $1$'s and $-1$'s.  Hence $n$ must be even.

If $r\ge3$, choose any three distinct columns of \(X\in\mathcal F_{\rm BO}\), 
relabel them as 
\(X_{\cdot1}\), \(X_{\cdot2}\), and \(X_{\cdot3}\), and set
$D=[X_{\cdot 1}\circ X_{\cdot 2},\ X_{\cdot 1}\circ X_{\cdot 3}]$.
Since the columns of $X$ are pairwise orthogonal,
$\sum_{i=1}^nD_{i1}=X_{\cdot 1}^\top X_{\cdot 2}=0$ and
$\sum_{i=1}^nD_{i2}=X_{\cdot 1}^\top X_{\cdot 3}=0$.
Moreover, since $X_{i1}^2=1$ for $i=1,\ldots,n$, we have
\[
 D_{\cdot 1}^\top D_{\cdot 2}
 =\sum_{i=1}^nX_{i1}^2X_{i2}X_{i3}
 =X_{\cdot 2}^\top X_{\cdot 3}
 =0.
\]
Thus, all first- and second-order $J$-characteristics of $D$ vanish.
By \eqref{eq:oa_j_characterization},
$D$ is an $\operatorname{OA}(n,2,2,2)$.  Each of its four level
combinations therefore occurs $n/4$ times, so $4\mid n$.
\end{proof}

\begin{remark}
When $r=n$, the set $\mathcal F_{\rm BO}$ consists precisely of the
Hadamard matrices of order $n$.  The cases $n=1$ and $n=2$ are trivial,
whereas Proposition~\ref{prop:bo_feasibility_conditions} shows that
$4\mid n$ is necessary for $n\ge3$.  This divisibility condition cannot
be strengthened merely by assuming $r\ge4$: a Hadamard matrix of order
four gives $\mathcal F_{\rm BO}\ne\emptyset$ for $(n,r)=(4,4)$, although
$8\nmid n$ \cite{hedayat1978hadamard}.  The Hadamard conjecture, which 
remains open \cite{ferber2022hadamard}, asserts
that a Hadamard matrix of order $4m$ exists for every positive integer
$m$.
\end{remark}

Throughout the remainder of the paper, we impose the following structural condition.

\begin{assumption}
\label{ass:bo_structure}
Either $r=1$ or $n$ is even.
\end{assumption}

\section{Error Bounds, Rounding, and Exact Penalization}
\label{sec:error_bound}

We establish residual-based rounding guarantees and local and global
error bounds, and derive exact-penalty results.

Existing error bounds address nonnegative and sign-constrained
orthogonality systems \cite{jiang2023exact,qian2024error,chen2024tight}.
For binary orthogonal optimization, the formulation in
\cite{xiao2026relaxation} keeps orthogonality exact and penalizes box
violation.  Our analysis keeps $X\in\Omega$ and uses
$F(X)=\|X^\top X-nI_r\|_F$ as the residual.  Moreover, our result has the optimal
Lipschitz exponent $1$.  For the general sign-constrained Stiefel manifold,
the corresponding exponent is $1/2$, except in special cases where it improves
to $1$ \cite{chen2024tight}.

Throughout the remainder of the paper, $L_f$ denotes a Lipschitz constant
of $f$ on $\Omega$ with respect to the Frobenius norm.  

\subsection{Residual gaps and rounding guarantees}

For $X\in\Omega$, define
\[
 a_j(X):=n-\|X_{\cdot j}\|_2^2=-C(X)_{jj}\ge0.
\]
The following elementary estimate bounds the distance to the sign-rounded
matrix.

\begin{lemma}
\label{lem:deficit_rounding}
Let $X\in\Omega$ have no zero entries and set
$\widehat X=\sign(X)$.  Then
\[
 \|X_{\cdot j}-\widehat X_{\cdot j}\|_2
 \le\|X_{\cdot j}-\widehat X_{\cdot j}\|_1
 \le a_j(X),\qquad j=1,\ldots,r,
\]
and consequently
\begin{equation}\label{eq:deficit_rounding}
   \|X-\widehat X\|_F\le F(X).
\end{equation} 
\end{lemma}

\begin{proof}
For $t\in[-1,1]\setminus\{0\}$,
$|t-\sign(t)|=1-|t|\le1-t^2$.  Applying this inequality columnwise gives
\[
 \|X_{\cdot j}-\widehat X_{\cdot j}\|_2
 \le\|X_{\cdot j}-\widehat X_{\cdot j}\|_1
 \le a_j(X).
\]
Squaring $a_j(X)$ and summing over $j$, and using
$\sum_j a_j(X)^2\le\|C(X)\|_F^2$ proves \eqref{eq:deficit_rounding}.
\end{proof}

We next give conditions ensuring that
$\sign(X)\in\mathcal F_{\rm BO}$.  Sign rounding can
fail in only two ways: some entry of $X$ is zero, in which case $\sign(X)$ is
not binary, or $\sign(X)$ is binary but its columns are not orthogonal.  The
following result gives a positive residual lower bound for each failure.

\begin{proposition}
\label{prop:sign_residual_gaps}
Let $X\in\Omega$.
\begin{enumerate}
 \item[i.] If $X$ has a zero entry, then $F(X)\ge1$.
 \item[ii.] If $X\in\mathcal B\setminus\mathcal F_{\rm BO}$, then
 \[
  F(X)\ge2\sqrt2.
 \]
 \item[iii.] Suppose that $X$ has no zero entries.  If
 $\widehat X=\sign(X)\notin\mathcal F_{\rm BO}$, then
 \begin{equation}\label{eq:nonorthogonal_sign_gap}
  F(X)\ge\sqrt{\frac85}.
 \end{equation}
\end{enumerate}
\end{proposition}

\begin{proof}
The first item follows immediately from
$a_j(X)=\sum_i(1-X_{ij}^2)\ge1$ for a column containing $0$.

For items (ii) and (iii), an infeasible binary sign pattern requires
$r\ge2$, so Assumption~\ref{ass:bo_structure} implies that $n$ is even.
For item (ii), binary infeasibility means that some second-order
$J$-characteristic
\(
 J_{\{p,q\}}(X)=X_{\cdot p}^\top X_{\cdot q}
\)
is nonzero with $p\neq q$.  Since $n$ is even, it is a nonzero even integer, and hence
$|X_{\cdot p}^\top X_{\cdot q}|\ge2$.  Moreover, because $X$ is binary,
$C(X)=X^\top X-nI_r$ is symmetric with zero diagonal.  Therefore,
\[
 F(X)^2
 =2\sum_{p<q}(X_{\cdot p}^\top X_{\cdot q})^2
 \ge 2\cdot 2^2=8.
\]

For item (iii), since
\(\widehat X=\sign(X)\) is binary but not orthogonal, there exist
distinct columns \(p\) and \(q\) such that
\(
 |J_{\{p,q\}}(\widehat X)|
 =|\widehat X_{\cdot p}^\top\widehat X_{\cdot q}|\ge2.
\)
Set $a=a_p(X)$, $b=a_q(X)$, and
$c=|X_{\cdot p}^\top X_{\cdot q}|$.  Using H\"older's inequality gives
\[
\begin{aligned}
 |\widehat X_{\cdot p}^\top\widehat X_{\cdot q}|
 &\le c
   +|(\widehat X_{\cdot p}-X_{\cdot p})^\top\widehat X_{\cdot q}|
   +|X_{\cdot p}^\top(\widehat X_{\cdot q}-X_{\cdot q})|\\
 &\le c
   +\|\widehat X_{\cdot p}-X_{\cdot p}\|_1
      \|\widehat X_{\cdot q}\|_\infty
   +\|X_{\cdot p}\|_\infty
      \|\widehat X_{\cdot q}-X_{\cdot q}\|_1\\
 &\le c
   +\|\widehat X_{\cdot p}-X_{\cdot p}\|_1
   +\|\widehat X_{\cdot q}-X_{\cdot q}\|_1
 \le a+b+c,
\end{aligned}
\]
where the last inequality follows from
Lemma~\ref{lem:deficit_rounding}.
By symmetry of $C(X)$,
$F(X)^2\ge a^2+b^2+2c^2$.  Using Cauchy--Schwarz yields
\[
 2\le a+b+c
 \le\sqrt{\frac52}\,(a^2+b^2+2c^2)^{1/2}
 \le\sqrt{\frac52}\,F(X).
\]
\end{proof}

Combining Proposition~\ref{prop:sign_residual_gaps} with \eqref{eq:deficit_rounding} 
yields an explicit residual threshold below
which sign rounding produces a matrix in \(\mathcal F_{\rm BO}\).

\begin{theorem}
\label{thm:finite_bo_identification}
If $X\in\Omega$ satisfies
\begin{equation}\label{eq:identification_threshold}
 F(X)<1,
\end{equation}
then $X$ has no zero entries and
\[
 \widehat X:=\sign(X)\in\mathcal F_{\rm BO}.
\]
Moreover,
\begin{equation}\label{eq:identification_distance}
 \dist(X,\mathcal F_{\rm BO})
 \le\|X-\widehat X\|_F
 \le F(X).
\end{equation}
\end{theorem}

\begin{proof}
Suppose that $F(X)<1$.  Proposition~\ref{prop:sign_residual_gaps} excludes
zero entries and nonorthogonal sign patterns.
Lemma~\ref{lem:deficit_rounding} then gives
\eqref{eq:identification_distance}.
\end{proof}

The threshold \(1\) in
\eqref{eq:identification_threshold} is sharp uniformly over the admissible
problem dimensions, with zero entries providing the limiting obstruction.
Indeed, when \(r=1\), a vector \(X\) with
one zero entry and all other entries in \(\{-1,1\}\) satisfies
\(F(X)=1\), while \(\sign(X)\) contains a zero entry and hence does not
belong to \(\mathcal F_{\rm BO}\).  By contrast, if \(X\) has no zero
entries, \eqref{eq:nonorthogonal_sign_gap} gives the
lower bound \(\sqrt{8/5}\) for the residual of a nonorthogonal sign pattern.

\begin{corollary}
\label{cor:rounding_objective_bound}
For any $X\in\Omega$ with $F(X)<1$, let
$\widehat X=\sign(X)$.  Then
\begin{equation}\label{eq:residual_objective_transfer}
 |f(\widehat X)-f(X)|
 \le L_fF(X).
\end{equation}
\end{corollary}

\begin{proof}
By Theorem~\ref{thm:finite_bo_identification},
$\|\widehat X-X\|_F\le F(X)$.  Hence
\[
 |f(\widehat X)-f(X)|
 \le L_f\|\widehat X-X\|_F
 \le L_fF(X).
\]
\end{proof}

For completeness, the following auxiliary certificate replaces the
Frobenius residual by its entrywise maximum norm.  The threshold $2/3$ for
$r\ge2$ is sufficient; no sharpness claim is made for that case.

\begin{corollary}
\label{cor:max_residual_certificate}
Let $X\in\Omega$.  If
\begin{equation}\label{eq:max_identification_threshold}
 \|C(X)\|_{\max}<
 \begin{cases}
  1,&r=1,\\
  2/3,&r\ge2,
 \end{cases}
\end{equation}
then $\sign(X)\in\mathcal F_{\rm BO}$, and
\[
 \dist(X,\mathcal F_{\rm BO})
 \le\|X-\sign(X)\|_F
 \le\sqrt r\,\|C(X)\|_{\max}.
\]
\end{corollary}

\begin{proof}
Set $\widehat X=\sign(X)$.  If \(X\) had a zero entry, then
\(
 a_j(X)=|C(X)_{jj}|\ge1
\)
for some $j$, contradicting \eqref{eq:max_identification_threshold} because
\(\|C(X)\|_{\max}<1\).  Hence \(\widehat X\in\mathcal B\).  If \(r=1\), then
\(\widehat X\in\mathcal F_{\rm BO}\).  If \(r\ge2\), suppose that
\(\widehat X\notin\mathcal F_{\rm BO}\).  Following the proof of
Proposition~\ref{prop:sign_residual_gaps}, there exist distinct indices
\(p\) and \(q\) such that
\[
 2\le a_p(X)+a_q(X)+|C(X)_{pq}|
 \le3\|C(X)\|_{\max},
\]
which contradicts \(\|C(X)\|_{\max}<2/3\).
Finally, since \(\widehat X\in\mathcal F_{\rm BO}\),
Lemma~\ref{lem:deficit_rounding} gives
\[
 \|X-\widehat X\|_F^2
 \le\sum_{j=1}^r a_j(X)^2
 \le r\|C(X)\|_{\max}^2.
\]
\end{proof}

The threshold $1$ for $r=1$ is sharp: a vector with one zero entry and all other
entries in \(\{-1,1\}\) has \(\|C(X)\|_{\max}=1\), while its sign vector is not binary.

\subsection{Local and global error bounds}

We first sharpen the estimate in
Theorem~\ref{thm:finite_bo_identification} near a feasible point.
A complementary lower bound for the linearized residual along tangent
directions is given in Proposition~\ref{prop:one_sided_coercivity} in the
appendix.

\begin{proposition}
\label{prop:refined_local_eb}
Let $\bar X\in\mathcal F_{\rm BO}$ and $\eta\in[0,1)$.  If $X\in\Omega$
satisfies $\|X-\bar X\|_{\max}\le\eta$, then
\begin{equation}\label{eq:refined_local_eb}
 \dist(X,\mathcal F_{\rm BO})
 \le\|X-\bar X\|_F
 \le\frac{1}{2-\eta}\,F(X).
\end{equation}
Consequently,
\begin{equation}\label{eq:asymptotic_local_modulus}
 \limsup_{\substack{\Omega\ni X\to\bar X\\
                    X\notin\mathcal F_{\rm BO}}}
 \frac{\dist(X,\mathcal F_{\rm BO})}{F(X)}
 \le\frac12.
\end{equation}
\end{proposition}

\begin{proof}
Suppose that $\|X-\bar X\|_{\max}\le\eta$.  Set
$d_{ij}=1-\bar X_{ij}X_{ij}$.  Then $0\le d_{ij}\le\eta$ and
\[
 1-X_{ij}^2=d_{ij}(2-d_{ij})\ge(2-\eta)d_{ij}.
\]
Hence
\[
 \|X_{\cdot j}-\bar X_{\cdot j}\|_2
 =\left(\sum_i d_{ij}^2\right)^{1/2}
 \le\sum_i d_{ij}
 \le\frac{a_j(X)}{2-\eta}.
\]
Squaring and summing over the columns proves
\eqref{eq:refined_local_eb}.  Next set
$\eta=\|X-\bar X\|_{\max}$ and let $X\to\bar X$ to obtain
\eqref{eq:asymptotic_local_modulus}.
\end{proof}

\begin{remark}
The asymptotic constant $1/2$ is sharp uniformly over the admissible problem
dimensions.  To see this, take $r=1$ and change one component of a binary
vector from $\bar X_i$ to $(1-t)\bar X_i$.  The distance to the binary set is
$t$, whereas $F(X)=2t-t^2$, so the ratio tends to $1/2$ as $t\downarrow0$.
\end{remark}

Proposition~\ref{prop:refined_local_eb} gives the local modulus for the
Lipschitz exponent $q=1$.  The next theorem gives a global error bound and
characterizes all valid H\"older exponents.

\begin{theorem}
\label{thm:global_error_bound}
Assume that $\mathcal F_{\rm BO}\ne\emptyset$.  Then
\begin{equation}\label{eq:explicit_global_eb}
 \dist(X,\mathcal F_{\rm BO})
 \le 2\sqrt{nr}\,F(X)
 \qquad\forall X\in\Omega.
\end{equation}
Moreover, for every \(q\in(0,1]\),
\[
 \dist(X,\mathcal F_{\rm BO})
 \le 2\sqrt{nr}\,(nr)^{1-q}F(X)^q
 \qquad \forall X\in\Omega.
\]
No error bound with exponent $q>1$ can hold, even locally around a feasible
point.
\end{theorem}

\begin{proof}
For $F(X)<1$, Theorem~\ref{thm:finite_bo_identification} gives the stronger estimate
$\dist(X,\mathcal F_{\rm BO})\le F(X)$.  For $F(X)\ge1$,
\eqref{eq:explicit_global_eb} follows from the box diameter
$\dist(X,\mathcal F_{\rm BO})\le2\sqrt{nr}$.  This proves the global
error bound for $q=1$.

If $0<q<1$, each entry of $C(X)$ has absolute value at most $n$, so
$F(X)\le nr$ on $\Omega$ and hence
\(
 F(X)\le(nr)^{1-q}F(X)^q.
\)
Thus the global error bound holds for every $0<q\le1$.
Conversely, fix $\bar X\in\mathcal F_{\rm BO}$ and set
$X(t)=(1-t)\bar X$.  For all
sufficiently small $t>0$, finiteness of $\mathcal F_{\rm BO}$ gives
$ \dist(X(t),\mathcal F_{\rm BO})=t\sqrt{nr}$, and 
$ F(X(t))=n\sqrt r\,(2t-t^2)$. 
A local bound with $q>1$ would imply $t=O(t^q)$ as $t\downarrow0$, which is
impossible.  Hence neither the local nor the global error bound can hold for
$q>1$.
\end{proof}

\subsection{Exact penalization}

We now use the error bounds to move the orthogonality equation into the
objective while retaining the box constraint.  Define
\[
 \underline f:=\min_{X\in\Omega}f(X).
\]
When $\mathcal F_{\rm BO}\ne\emptyset$, also define
\[
 f^*:=\min_{X\in\mathcal F_{\rm BO}}f(X).
\]
For $\mu>0$, consider the penalty problem
\begin{equation}\label{eq:exact_penalty_model}
 \min_{X\in\Omega}  P_\mu(X):=f(X)+\mu F(X).
\end{equation}

\begin{theorem}
\label{thm:global_exact_penalty}
Assume that $\mathcal F_{\rm BO}\ne\emptyset$ and
\begin{equation}\label{eq:global_penalty_threshold}
 \mu>\max\{L_f,f^*-\underline f\}.
\end{equation}
Then
\(
 \arg\min_{X\in\Omega}P_\mu(X)
 =\arg\min_{X\in\mathcal F_{\rm BO}}f(X).
\)
\end{theorem}

\begin{proof}
Every global solution of \eqref{eq:BO} has penalized value $f^*$.  Therefore,
it suffices to show that $P_\mu(X)>f^*$ whenever
$X\in\Omega\setminus\mathcal F_{\rm BO}$.  If $0<F(X)<1$, then
Corollary~\ref{cor:rounding_objective_bound} gives
$\widehat X:=\sign(X)\in\mathcal F_{\rm BO}$ and
$f(X)\ge f(\widehat X)-L_fF(X)$.  Hence
\[
 P_\mu(X)
 \ge f(\widehat X)+(\mu-L_f)F(X)>f^*.
\]
If $F(X)\ge1$, then
\[
 P_\mu(X)\ge\underline f+\mu>f^*.
\]
\end{proof}
 
The threshold in \eqref{eq:global_penalty_threshold} depends on the
unknown constrained optimal value $f^*$.  The fully explicit choice
\begin{equation}\label{eq:explicit_global_penalty_threshold}
   \mu>2L_f\sqrt{nr},
\end{equation} 
is sufficient.  Indeed, Lipschitz continuity and the diameter of $\Omega$ give
\[
 f^*-\underline f
 \le\max_{X\in\Omega}f(X)-\min_{X\in\Omega}f(X)
 \le2L_f\sqrt{nr}.
\]
Since $2\sqrt{nr}\ge1$, this choice also ensures $\mu>L_f$ and hence
implies \eqref{eq:global_penalty_threshold}.

The asymptotic modulus in Proposition~\ref{prop:refined_local_eb} yields the
smaller local threshold $\mu>L_f/2$ below, and this threshold is uniformly sharp.

\begin{proposition}
\label{prop:sharp_local_penalty}
Let $\bar X\in\mathcal F_{\rm BO}$.  If $\mu>L_f/2$, then $\bar X$ is a
strict local minimizer of $P_\mu$ over $\Omega$.
\end{proposition}

\begin{proof}
Choose $\eta\in(0,1)$ sufficiently small that
$\mu>L_f/(2-\eta)$.  For every
$X\in\Omega$ with $0<\|X-\bar X\|_{\max}\le\eta$,
Proposition~\ref{prop:refined_local_eb} gives
\[ 
 P_\mu(X)-P_\mu(\bar X)
 \ge-L_f\|X-\bar X\|_F+\mu F(X) 
 \ge\left(\mu-\frac{L_f}{2-\eta}\right)F(X)>0. 
\]
Because $\eta<1$, $\bar X$ is the only binary matrix in this neighborhood.
Thus $X\ne\bar X$ implies $F(X)>0$, which makes the last inequality strict.
\end{proof}

\begin{remark}
The threshold $L_f/2$ is sharp uniformly over the admissible problem
dimensions.  Take $r=1$, fix $\bar X\in\{-1,1\}^n$, and set
$f(X)=L_f\bar X_iX_i$ for one index $i$.  Change only that component by
$X_i(t)=(1-t)\bar X_i$.  Then
\[
 P_\mu(X(t))-P_\mu(\bar X)
 =(-L_f+2\mu)t-\mu t^2<0
\]
for every sufficiently small $t>0$ whenever $\mu\le L_f/2$.
\end{remark}

\section{A Smoothing Penalty Method for Binary Orthogonal Optimization}
\label{sec:alg_framework}
 
The exact penalty $P_\mu$ is generally nonsmooth at points where $C(X)=0$.  To obtain smooth
box-constrained subproblems, let $\delta>0$ and define the Huber--Moreau
smoothing $h_\delta:\mathbb S^r\to\mathbb R$ of the Frobenius norm by
\begin{equation}\label{eq:huber_residual}
 h_\delta(S):=
 \begin{cases}
  \|S\|_F^2/(2\delta),&\|S\|_F\le\delta,\\
  \|S\|_F-\delta/2,&\|S\|_F>\delta.
 \end{cases}
\end{equation}
The gradient of $h_\delta$ is the mapping
\(
\nabla h_\delta(S)
 =\frac{S}{\max\{\delta,\|S\|_F\}}
\) for $S\in\mathbb S^r$.
The smoothed penalty objective is
\begin{equation}\label{eq:smoothed_penalty_model}
 P_{\mu,\delta}(X):=f(X)+\mu h_\delta(C(X)),
 \qquad X\in\Omega.
\end{equation}
For $\mu,\delta>0$, the chain rule gives
\begin{equation}\label{eq:smoothed_penalty_gradient}
 \nabla P_{\mu,\delta}(X)
 =\nabla f(X)+2\mu X\nabla h_\delta(C(X)).
\end{equation}
Moreover,
\begin{equation}\label{eq:smoothing_uniform_error}
 0\le\|S\|_F-h_\delta(S)\le\frac\delta2,
\end{equation}
and hence
\begin{equation}\label{eq:smoothed_objective_error}
 0\le P_\mu(X)-P_{\mu,\delta}(X)\le\frac{\mu\delta}{2},
 \qquad X\in\Omega.
\end{equation}
Thus, $P_{\mu,\delta}\to P_\mu$ uniformly on $\Omega$ as
$\delta\downarrow0$.
For a function \(\varphi\) continuously differentiable on a neighborhood
of \(\Omega\) and
\(\alpha>0\), define the gradient mapping
\[
 \mathcal G_\alpha(X;\varphi)
 :=\frac1\alpha\bigl[
 X-\Pi_\Omega(X-\alpha\nabla\varphi(X))\bigr].
\]
For every $\alpha>0$, $\mathcal G_\alpha(X;\varphi)=0$ is equivalent to
$0\in\nabla\varphi(X)+N_\Omega(X)$.  We use
$\|\mathcal G_1(X;\varphi)\|_F$ to measure stationarity.
For fixed \(\mu,\delta>0\), a projected
gradient step for \(P_{\mu,\delta}\) is
\begin{equation}\label{eq:penalty_pg_step}
 X^+=X-\alpha\mathcal G_\alpha(X;P_{\mu,\delta})
 =\Pi_\Omega\bigl(X-\alpha\nabla P_{\mu,\delta}(X)\bigr).
\end{equation}
The projection is componentwise clipping, and the penalty-gradient term in
\eqref{eq:smoothed_penalty_gradient} requires only products involving
\(X\) and the \(r\times r\) matrix \(C(X)\).
 
Once $X\in\Omega$ satisfies
$F(X)<1$, Theorem~\ref{thm:finite_bo_identification} gives a feasible
binary matrix $B=\sign(X)$.
To improve its objective value while
preserving feasibility, we use the following balanced-switch neighborhood.
For $B\in\mathcal F_{\rm BO}$, a column $p$, and a nonempty index set
$T\subseteq\{1,\ldots,n\}$, let $B^{p,T}$ be obtained by replacing $B_{ip}$
with $-B_{ip}$ for every $i\in T$ and leaving all other entries unchanged.
We call $T$ balanced for column $p$ if
\begin{equation}\label{eq:balanced_switch_condition}
 \sum_{i\in T}B_{ip}B_{iq}=0
 \qquad \forall q\ne p.
\end{equation}
Indeed, for every $q\ne p$,
\(
 \bigl[(B^{p,T})_{\cdot p}\bigr]^\top(B^{p,T})_{\cdot q}
 =B_{\cdot p}^\top B_{\cdot q}
 -2\sum_{i\in T}B_{ip}B_{iq}=0,
\)
while all other column inner products remain unchanged.  Hence
$B^{p,T}\in\mathcal F_{\rm BO}$.

For a prescribed integer $m\ge1$, define
\[
 \mathcal N_m(B):=
 \left\{
 B^{p,T}:1\le p\le r,\ 1\le|T|\le m,\
 T\ \text{satisfies \eqref{eq:balanced_switch_condition}}
 \right\}.
\]
Starting from $B^0=\sign(X)$, the balanced-switch local search (BSLS)
replaces the current point by a strictly improving point in
$\mathcal N_m(B)$ and stops when no such point exists.

We next combine the smooth subproblem solver with bounded penalty and
smoothing continuation.
Choose a finite target penalty \(\bar\mu>0\), an initial value
\(\mu_0\in(0,\bar\mu]\), a growth factor \(\gamma_\mu>1\), and smoothing
parameters \(\delta_k\downarrow0\).  The penalty sequence
\(
 \mu_k:=\min\{\bar\mu,\gamma_\mu^k\mu_0\}
\)
reaches \(\bar\mu\) after finitely many stages and remains fixed thereafter.
Although $P_{\bar\mu,\delta}$ is generally not exact for fixed
$\delta>0$, its limiting nonsmooth model $P_{\bar\mu}$ is globally exact
whenever $\bar\mu$ satisfies
\eqref{eq:global_penalty_threshold}. 
A sufficient condition is $\bar\mu>2L_f\sqrt{nr}$ by 
\eqref{eq:explicit_global_penalty_threshold}.

\begin{algorithm}[t]
\caption{SPBO with balanced-switch local refinement}
\label{alg:spbo}
\begin{algorithmic}[1]
\STATE Choose \(X^0\in\Omega\), \(0<\mu_0\le\bar\mu\),
\(\gamma_\mu>1\), an integer $m\ge1$, and sequences
\(\delta_k>0\) and \(\varepsilon_k>0\), both decreasing to zero.
\FOR{\(k=0,1,2,\ldots\)}
 \STATE Set \(\mu_k=\min\{\bar\mu,\gamma_\mu^k\mu_0\}\).
 \STATE Starting from \(X^{k,0}=X^k\), apply projected-gradient steps
 \eqref{eq:penalty_pg_step} to \(P_{\mu_k,\delta_k}\) until an iterate
 \(X^{k+1}\in\Omega\) satisfies
 \[
  \|\mathcal G_1(X^{k+1};P_{\mu_k,\delta_k})\|_F
  \le\varepsilon_k.
 \]
 \IF{\(F(X^{k+1})<1\)}\label{step:spbo_certificate}
   \STATE Set $B^0=\sign(X^{k+1})$ and $B=B^0$.
   \WHILE{there exists $B'\in\mathcal N_m(B)$ with $f(B')<f(B)$}
    \STATE Choose
    $B^+\in\arg\min\{f(B'):B'\in\mathcal N_m(B)\}$ and set $B=B^+$.
   \ENDWHILE
   \STATE Return $X^{k+1}$, $B^0$, and $B^{\rm out}=B$.
 \ENDIF
\ENDFOR
\end{algorithmic}
\end{algorithm}

By Theorem~\ref{thm:finite_bo_identification},
the test in Step~\ref{step:spbo_certificate} guarantees
$B^0=\sign(X^{k+1})\in\mathcal F_{\rm BO}$.
The rounding bound and the monotonicity of
BSLS give
\[
 f(B^{\rm out})
 \le f(B^0)
 \le f(X^{k+1})+L_fF(X^{k+1}).
\]
Since BSLS moves among a finite set of binary matrices and every accepted
move strictly decreases $f$, no iterate can be revisited.  Hence BSLS
terminates after finitely many accepted moves.
Moreover, $B^{\rm out}$ is locally optimal with respect to
$\mathcal N_m$.  The refinement is invoked only after the residual
certificate is satisfied and therefore does not affect the smooth-subproblem
complexity or the stationarity analysis of a nonterminating continuation
sequence.

Let $X^{k,\ell}$ denote the $\ell$th inner iterate at stage $k$,
with $X^{k,0}=X^k$.

\begin{proposition}
\label{prop:smoothed_subproblem_complexity}
Suppose that $\nabla f$ is $L_{\nabla f}$-Lipschitz on $\Omega$.  Fix a
stage $k$, write $\mu=\mu_k>0$ and $\delta=\delta_k>0$, and set
\[
 L_{\mu,\delta}
 :=\max\left\{1,L_{\nabla f}+2\mu\left(1+\frac{2nr}{\delta}\right)\right\}.
\]
Let $\{X^{k,\ell}\}$ be generated by \eqref{eq:penalty_pg_step} with
$\alpha=1/L_{\mu,\delta}$.  Then, for every $N\ge1$,
\begin{equation}\label{eq:pg_complexity}
 \min_{0\le\ell<N}\|\mathcal G_1(X^{k,\ell};P_{\mu,\delta})\|_F
 \le
 \sqrt{\frac{2L_{\mu,\delta}
 [P_{\mu,\delta}(X^{k,0})-P_{\mu,\delta}^{\inf}]}{N}},
\end{equation}
where $P_{\mu,\delta}^{\inf}:=\min_\Omega P_{\mu,\delta}$.  Consequently,
\[
 O\!\left(
 \frac{L_{\mu,\delta}
 [P_{\mu,\delta}(X^{k,0})-P_{\mu,\delta}^{\inf}]}
 {\varepsilon^2}
 \right)
\]
projected-gradient iterations suffice to produce a point $X\in\Omega$ with
$\|\mathcal G_1(X;P_{\mu,\delta})\|_F\le\varepsilon$.
\end{proposition}

\begin{proof}
The mapping
$\nabla h_\delta(S)=S/\max\{\delta,\|S\|_F\}$ is the projection of $S/\delta$
onto the unit Frobenius ball.  Hence it is $1/\delta$-Lipschitz and
$\|\nabla h_\delta(S)\|_F\le1$.  For $X,Y\in\Omega$,
\[
 \|C(X)-C(Y)\|_F
 \le(\|X\|_2+\|Y\|_2)\|X-Y\|_F
 \le2\sqrt{nr}\,\|X-Y\|_F,
\]
and therefore
\[
 \begin{aligned}
 &\|X\nabla h_\delta(C(X))-Y\nabla h_\delta(C(Y))\|_F\\
 &\quad\le\|X-Y\|_F
 +\|Y\|_2\|\nabla h_\delta(C(X))-\nabla h_\delta(C(Y))\|_F\\
 &\quad\le\left(1+\frac{2nr}{\delta}\right)\|X-Y\|_F.
 \end{aligned}
\]
Thus $L_{\mu,\delta}$ is a Lipschitz constant for
$\nabla P_{\mu,\delta}$.  The projected-gradient descent estimate gives
\[
 P_{\mu,\delta}(X^{k,\ell+1})
 \le P_{\mu,\delta}(X^{k,\ell})
 -\frac{1}{2L_{\mu,\delta}}
 \|\mathcal G_{1/L_{\mu,\delta}}
 (X^{k,\ell};P_{\mu,\delta})\|_F^2.
\]
Since $L_{\mu,\delta}\ge1$, the standard stepsize comparison for
projected-gradient mappings gives
\[
 \|\mathcal G_1(X;P_{\mu,\delta})\|_F
 \le\|\mathcal G_{1/L_{\mu,\delta}}
 (X;P_{\mu,\delta})\|_F.
\]
Combining this estimate with the descent inequality and summing over
$\ell=0,\ldots,N-1$ proves \eqref{eq:pg_complexity} and the iteration bound.
\end{proof}

For fixed $\mu>0$, $L_{\mu,\delta}=O(\delta^{-1})$ as
$\delta\downarrow0$.  In the implementation, a safeguarded spectral step
and a nonmonotone Armijo line search are used instead of the conservative
fixed step \cite{Grippo86,Grippo02}.

If the test in Step~\ref{step:spbo_certificate} of
Algorithm~\ref{alg:spbo} is never satisfied, the algorithm generates an
infinite continuation sequence.  The following example shows that this
alternative cannot be excluded in general.

\begin{example}
Suppose that $\nabla f(0)=0$.  Since the gradient of the smoothed
penalty term contains the factor $X$, we have
$\nabla P_{\mu,\delta}(0)=0$ for every $\mu,\delta>0$.
If $X^0=0$, every projected-gradient step remains at zero, whereas
\[
 F(0)=n\sqrt r\ge1.
\]
This same point violates the global error-bound nondegeneracy condition
used in the constraint-dissolving exact-penalty framework
\cite[Assumption~3.26]{xiao2025exactpenalty}: indeed,
$DC(0)^*[C(0)]=0$ and $N_\Omega(0)=\{0\}$, whereas $F(0)>0$.
Consequently, the test in Step~\ref{step:spbo_certificate} is never
satisfied.
\end{example}

A point
\(X^*\in\Omega\) is stationary for \(P_\mu\) if there exists
\(V^*\in\mathbb S^r\) such that
\begin{equation}\label{eq:exact_penalty_stationarity}
 0\in\nabla f(X^*)+2\mu X^*V^*+N_\Omega(X^*),
\end{equation}
where
\begin{equation}\label{eq:norm_subgradient}
 V^*=
 \frac{C(X^*)}{F(X^*)}\quad\text{if }F(X^*)>0,
 \qquad
 \|V^*\|_F\le1\quad\text{if }F(X^*)=0.
\end{equation}
Here $\partial\|\cdot\|_F$ denotes the convex subdifferential of the
Frobenius norm on $\mathbb S^r$.  Equivalently,
\(V^*\in\partial\|\cdot\|_F(C(X^*))\), and
\eqref{eq:exact_penalty_stationarity} is the composite first-order
stationarity condition for \(P_\mu\) over \(\Omega\).

\begin{theorem}
\label{thm:spbo_stationarity}
Let $\bar\mu>0$, $\delta_k>0$ with $\delta_k\downarrow0$, and
$\varepsilon_k\downarrow0$.
Suppose that $\{X^{k+1}\}\subset\Omega$ satisfies
\[
 \|\mathcal G_1(X^{k+1};P_{\bar\mu,\delta_k})\|_F
 \le\varepsilon_k
\]
for all sufficiently large $k$.  Then every accumulation point of
\(\{X^{k+1}\}\) is stationary for
\(P_{\bar\mu}\) in the sense of
\eqref{eq:exact_penalty_stationarity}--\eqref{eq:norm_subgradient}.
\end{theorem}

\begin{proof}
Let a subsequence, indexed again by \(k\), converge to \(X^*\), and set
\(
 V^k:=\nabla h_{\delta_k}(C(X^{k+1})).
\)
Since \(\|V^k\|_F\le1\), a further subsequence converges to some \(V^*\).
If \(F(X^*)>0\), continuity gives
\(V^*=C(X^*)/F(X^*)\); if \(F(X^*)=0\), the bound
\(\|V^*\|_F\le1\) remains.  Thus \eqref{eq:norm_subgradient} holds.
Let
\(
 Z^k:=\Pi_\Omega\bigl(
 X^{k+1}-\nabla P_{\bar\mu,\delta_k}(X^{k+1})\bigr).
\)
The residual condition gives \(Z^k-X^{k+1}\to0\).  The characterization of
Euclidean projection yields
\[
 0\in\nabla P_{\bar\mu,\delta_k}(X^{k+1})
 +(Z^k-X^{k+1})+N_\Omega(Z^k).
\]
Using \eqref{eq:smoothed_penalty_gradient}, taking limits, and invoking the
closedness of the normal-cone graph give
\[
 0\in\nabla f(X^*)+2\bar\mu X^*V^*+N_\Omega(X^*).
\]
\end{proof}
 
\section{Numerical Experiments}
\label{sec:numerical}

We evaluate \alg{SPBO} through four questions arising from the preceding
analysis.  The first experiment examines the residual thresholds used by the
rounding guarantees.  The remaining experiments test the penalty method on
orthogonal binary projection, supervised hashing, and Hadamard completion.
All experiments were run in MATLAB on a Windows workstation with an
Intel\textregistered\ Core\texttrademark\ i9-12900HX CPU and 32\,GB RAM.
For the three application experiments, SPBO uses
$\varepsilon_{\rm opt}=\varepsilon_{\rm feas}=10^{-3}$, at most $200$
continuation stages and $250$ inner iterations per stage, automatically
initialized $\mu_0$ with balance factor $1$, and $\gamma_\mu=1.5$.
At stage $k$, the inner solver uses the stopping criterion
\[
 \|\mathcal G_1(X;P_{\mu_k,\delta_k})\|_{\max}
 \le \max\{\varepsilon_{\rm opt},
          \min\{10^{-2},\sqrt{\delta_k}\}\},
\]
subject to the prescribed iteration limit.
The feasibility tolerance $\varepsilon_{\rm feas}$ is applied to
the continuous residual $F(X)$.
Across these experiments, MGBO uses $\gamma=0.2$ and $\rho=10^{-2}$.
For projection and Hadamard completion, its penalty parameter is $90$,
while AM uses $\rho=50$ and ABMO uses $\alpha=0.5$.
We report objective value or retrieval
quality, orthogonality violation, and CPU time separately.  For SPBO, the
reported binary output is obtained by sign rounding the final box-constrained
iterate.  Both the objective value and orthogonality violation are evaluated
at this rounded output.


\subsection{Error-Bound and Rounding Certificates}
\label{subsec:error_bound_experiment}

This experiment tests the feasibility certificates and rounding-distance
bounds of Theorem~\ref{thm:finite_bo_identification} and
Corollary~\ref{cor:max_residual_certificate} along SPBO trajectories.
We use $20$ Gaussian projection instances with $n=64$ and $r=8$.
Here $\varepsilon_{\rm opt}=\varepsilon_{\rm feas}=10^{-8}$,
with at most $30$ continuation stages and $12$ inner iterations per stage.
The penalty initialization is automatic with balance factor $1$ and
$\gamma_\mu=1.5$; we set $\delta_0=10^{-2}$, reduce $\delta_k$ by a factor
of $0.5$ per stage, and continue recording iterates after sign rounding
becomes feasible.
Figure~\ref{fig:error_bound_certificates} displays the
Frobenius- and maximum-residual trajectories, rounding distances, rounded
orthogonality residuals, and successful-rounding percentages.

\begin{figure}[H]
\centering
\subfigure[Frobenius-residual certificate.]{%
\includegraphics[width=0.96\textwidth]{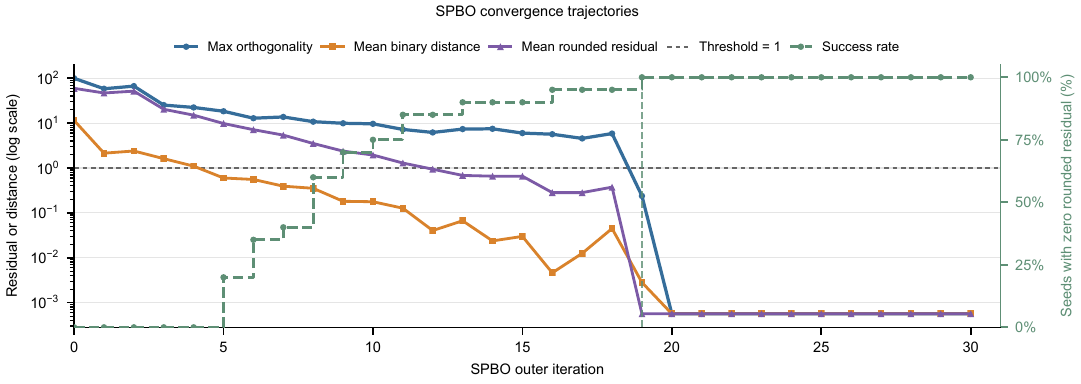}%
\label{fig:error_bound_frobenius_certificate}}
\par\medskip
\subfigure[Maximum-residual certificate.]{%
\includegraphics[width=0.96\textwidth]{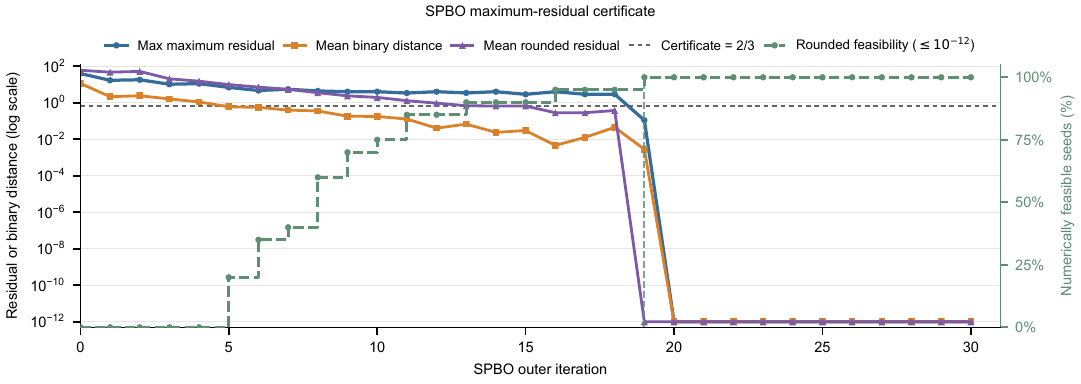}%
\label{fig:error_bound_max_certificate}}
\caption{Residual-certificate trajectories for 20 runs with $n=64$ and
$r=8$.}
\label{fig:error_bound_certificates}
\end{figure}


\subsection{Orthogonal Binary Projection}
\label{subsec:proj}

To examine how the penalty method reduces orthogonality violation in a
continuous data-fitting problem,
we consider
\begin{equation}
\min_{X\in\{-1,+1\}^{n\times r}} \|X-Z\|_F^2
\quad\text{s.t.}\quad X^\top X=nI_r,
\label{eq:proj}
\end{equation}
where the entries of $Z\in\mathbb R^{n\times r}$ are sampled independently
from $\mathcal N(0,1)$.  The row dimensions are 500, 1000, 2000, 5000, and
8000; the ranks are 16, 32, and 64; and each configuration is evaluated in
ten independent trials.  The comparison includes
\alg{MGBO}, \alg{Relax}, \alg{AM}, \alg{ABMO}, and \alg{SPBO}.
MGBO, AM, and ABMO use a maximum of $2000$ iterations and a stopping
tolerance of $10^{-3}$; MGBO uses $50$ subiterations.
Figure~\ref{fig:proj_scaling} and Table~\ref{tab:proj_orth_performance_all}
report the mean performance over ten independent trials.  SPBO has the smallest mean orthogonality
violation in all 15 $(n,r)$ configurations.  Compared with the next-smallest
mean violation, it is lower by factors ranging from $3.24$ to
$1.04\times10^2$.  Its objective ranks second in eight configurations and
third in seven, whereas its CPU time is largest in all 15 configurations.

\begin{figure}[tbp]
\centering
\includegraphics[width=\textwidth]{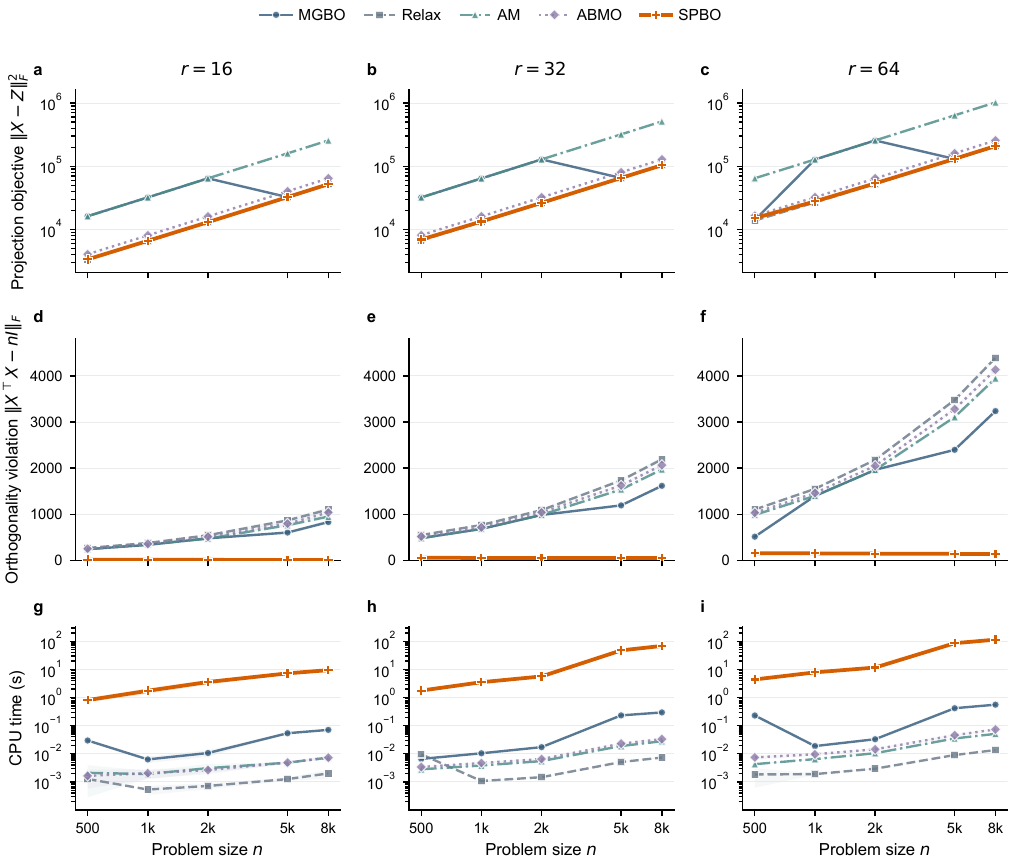}
\caption{Squared projection objective $\|X-Z\|_F^2$, orthogonality violation,
and CPU time over ten independent trials. Curves show sample means; shaded
bands indicate two-sided $95\%$ confidence intervals.}
\label{fig:proj_scaling}
\end{figure}

\begin{sidewaystable}[p]
\centering
\scriptsize
\setlength{\tabcolsep}{3.2pt}
\renewcommand{\arraystretch}{0.82}
\caption{Performance comparison on randomly generated orthogonal binary projection problems for $r\in\{16,32,64\}$. The objective is $\|X-Z\|_F^2$.}
\label{tab:proj_orth_performance_all}
\vspace{-0.45em}
\noindent\makebox[\linewidth][c]{%
\resizebox{0.96\linewidth}{!}{%
\begin{tabular}{@{}cccccc@{}}
\toprule
\multirow{2}{*}{$n$} & \multicolumn{5}{c}{Squared objective $\|X-Z\|_F^2$ ($r=16$)} \\ \cmidrule(lr){2-6}
& MGBO & Relax & AM & ABMO & SPBO \\
500 & 1.619e+04 $\pm$ 1.904e+02 & \textbf{3.308e+03 $\pm$ 5.263e+01} & 1.613e+04 $\pm$ 1.887e+02 & 4.036e+03 $\pm$ 7.595e+01 & 3.349e+03 $\pm$ 5.446e+01 \\
1000 & 3.209e+04 $\pm$ 2.941e+02 & \textbf{6.538e+03 $\pm$ 7.219e+01} & 3.210e+04 $\pm$ 2.938e+02 & 8.028e+03 $\pm$ 1.007e+02 & 6.587e+03 $\pm$ 7.372e+01 \\
2000 & 6.427e+04 $\pm$ 2.844e+02 & \textbf{1.296e+04 $\pm$ 7.578e+01} & 6.429e+04 $\pm$ 3.042e+02 & 1.601e+04 $\pm$ 1.240e+02 & 1.302e+04 $\pm$ 7.725e+01 \\
5000 & 3.238e+04 $\pm$ 7.431e+01 & \textbf{3.237e+04 $\pm$ 7.179e+01} & 1.598e+05 $\pm$ 6.545e+02 & 4.006e+04 $\pm$ 1.038e+02 & 3.242e+04 $\pm$ 7.387e+01 \\
8000 & 5.189e+04 $\pm$ 2.410e+02 & \textbf{5.188e+04 $\pm$ 2.401e+02} & 2.561e+05 $\pm$ 9.080e+02 & 6.419e+04 $\pm$ 3.409e+02 & 5.193e+04 $\pm$ 2.402e+02 \\
\midrule
\multirow{2}{*}{$n$} & \multicolumn{5}{c}{Orthogonality violation ($r=16$)} \\ \cmidrule(lr){2-6}
& MGBO & Relax & AM & ABMO & SPBO \\
500 & 2.440e+02 $\pm$ 2.342e+01 & 2.695e+02 $\pm$ 1.440e+01 & 2.435e+02 $\pm$ 2.282e+01 & 2.537e+02 $\pm$ 1.402e+01 & \textbf{1.763e+01 $\pm$ 2.880e+00} \\
1000 & 3.367e+02 $\pm$ 1.988e+01 & 3.780e+02 $\pm$ 1.339e+01 & 3.458e+02 $\pm$ 1.820e+01 & 3.588e+02 $\pm$ 2.273e+01 & \textbf{1.308e+01 $\pm$ 2.362e+00} \\
2000 & 4.797e+02 $\pm$ 3.465e+01 & 5.466e+02 $\pm$ 2.950e+01 & 4.728e+02 $\pm$ 4.017e+01 & 5.131e+02 $\pm$ 3.386e+01 & \textbf{1.350e+01 $\pm$ 4.719e+00} \\
5000 & 6.047e+02 $\pm$ 4.373e+01 & 8.678e+02 $\pm$ 6.270e+01 & 7.641e+02 $\pm$ 4.998e+01 & 7.969e+02 $\pm$ 6.200e+01 & \textbf{7.971e+00 $\pm$ 2.419e+00} \\
8000 & 8.318e+02 $\pm$ 6.221e+01 & 1.105e+03 $\pm$ 6.619e+01 & 9.511e+02 $\pm$ 6.816e+01 & 1.039e+03 $\pm$ 5.332e+01 & \textbf{8.029e+00 $\pm$ 4.364e+00} \\
\bottomrule
\end{tabular}
}
}\par
\vspace{0.08em}
\noindent\makebox[\linewidth][c]{%
\resizebox{0.96\linewidth}{!}{%
\begin{tabular}{@{}cccccc@{}}
\toprule
\multirow{2}{*}{$n$} & \multicolumn{5}{c}{Squared objective $\|X-Z\|_F^2$ ($r=32$)} \\ \cmidrule(lr){2-6}
& MGBO & Relax & AM & ABMO & SPBO \\
500 & 3.201e+04 $\pm$ 2.312e+02 & \textbf{6.688e+03 $\pm$ 6.786e+01} & 3.202e+04 $\pm$ 2.488e+02 & 8.101e+03 $\pm$ 8.353e+01 & 6.895e+03 $\pm$ 8.344e+01 \\
1000 & 6.410e+04 $\pm$ 3.658e+02 & \textbf{1.312e+04 $\pm$ 1.148e+02} & 6.415e+04 $\pm$ 3.543e+02 & 1.603e+04 $\pm$ 1.624e+02 & 1.332e+04 $\pm$ 1.116e+02 \\
2000 & 1.282e+05 $\pm$ 7.306e+02 & \textbf{2.612e+04 $\pm$ 8.569e+01} & 1.283e+05 $\pm$ 6.660e+02 & 3.214e+04 $\pm$ 5.666e+01 & 2.631e+04 $\pm$ 9.746e+01 \\
5000 & 6.504e+04 $\pm$ 1.569e+02 & \textbf{6.502e+04 $\pm$ 1.574e+02} & 3.197e+05 $\pm$ 1.258e+03 & 8.019e+04 $\pm$ 1.942e+02 & 6.521e+04 $\pm$ 1.398e+02 \\
8000 & 1.037e+05 $\pm$ 3.230e+02 & \textbf{1.036e+05 $\pm$ 3.230e+02} & 5.122e+05 $\pm$ 1.539e+03 & 1.281e+05 $\pm$ 3.544e+02 & 1.038e+05 $\pm$ 3.223e+02 \\
\midrule
\multirow{2}{*}{$n$} & \multicolumn{5}{c}{Orthogonality violation ($r=32$)} \\ \cmidrule(lr){2-6}
& MGBO & Relax & AM & ABMO & SPBO \\
500 & 4.840e+02 $\pm$ 1.678e+01 & 5.504e+02 $\pm$ 1.407e+01 & 4.823e+02 $\pm$ 1.378e+01 & 5.193e+02 $\pm$ 1.240e+01 & \textbf{5.636e+01 $\pm$ 1.984e+00} \\
1000 & 6.828e+02 $\pm$ 2.175e+01 & 7.684e+02 $\pm$ 2.073e+01 & 6.865e+02 $\pm$ 2.071e+01 & 7.175e+02 $\pm$ 2.480e+01 & \textbf{5.542e+01 $\pm$ 2.006e+00} \\
2000 & 9.859e+02 $\pm$ 3.587e+01 & 1.091e+03 $\pm$ 2.626e+01 & 9.878e+02 $\pm$ 3.675e+01 & 1.040e+03 $\pm$ 3.523e+01 & \textbf{5.387e+01 $\pm$ 1.069e+00} \\
5000 & 1.193e+03 $\pm$ 2.815e+01 & 1.734e+03 $\pm$ 4.586e+01 & 1.538e+03 $\pm$ 4.204e+01 & 1.624e+03 $\pm$ 4.281e+01 & \textbf{5.261e+01 $\pm$ 1.731e+00} \\
8000 & 1.617e+03 $\pm$ 5.033e+01 & 2.191e+03 $\pm$ 6.877e+01 & 1.972e+03 $\pm$ 3.525e+01 & 2.066e+03 $\pm$ 5.746e+01 & \textbf{5.220e+01 $\pm$ 1.837e+00} \\
\bottomrule
\end{tabular}
}
}\par
\vspace{0.08em}
\noindent\makebox[\linewidth][c]{%
\resizebox{0.96\linewidth}{!}{%
\begin{tabular}{@{}cccccc@{}}
\toprule
\multirow{2}{*}{$n$} & \multicolumn{5}{c}{Squared objective $\|X-Z\|_F^2$ ($r=64$)} \\ \cmidrule(lr){2-6}
& MGBO & Relax & AM & ABMO & SPBO \\
500 & 1.404e+04 $\pm$ 1.176e+02 & \textbf{1.380e+04 $\pm$ 1.162e+02} & 6.445e+04 $\pm$ 3.190e+02 & 1.637e+04 $\pm$ 1.550e+02 & 1.507e+04 $\pm$ 1.143e+02 \\
1000 & 1.279e+05 $\pm$ 6.314e+02 & \textbf{2.672e+04 $\pm$ 1.573e+02} & 1.279e+05 $\pm$ 7.189e+02 & 3.228e+04 $\pm$ 1.856e+02 & 2.770e+04 $\pm$ 1.261e+02 \\
2000 & 2.560e+05 $\pm$ 7.370e+02 & \textbf{5.256e+04 $\pm$ 1.734e+02} & 2.560e+05 $\pm$ 7.163e+02 & 6.408e+04 $\pm$ 2.727e+02 & 5.347e+04 $\pm$ 1.747e+02 \\
5000 & 1.303e+05 $\pm$ 3.740e+02 & \textbf{1.302e+05 $\pm$ 3.813e+02} & 6.397e+05 $\pm$ 1.452e+03 & 1.602e+05 $\pm$ 4.138e+02 & 1.310e+05 $\pm$ 3.845e+02 \\
8000 & 2.078e+05 $\pm$ 1.822e+02 & \textbf{2.077e+05 $\pm$ 1.903e+02} & 1.024e+06 $\pm$ 1.686e+03 & 2.562e+05 $\pm$ 2.581e+02 & 2.086e+05 $\pm$ 1.587e+02 \\
\midrule
\multirow{2}{*}{$n$} & \multicolumn{5}{c}{Orthogonality violation ($r=64$)} \\ \cmidrule(lr){2-6}
& MGBO & Relax & AM & ABMO & SPBO \\
500 & 5.148e+02 $\pm$ 8.128e+00 & 1.104e+03 $\pm$ 1.967e+01 & 9.869e+02 $\pm$ 8.106e+00 & 1.033e+03 $\pm$ 1.163e+01 & \textbf{1.589e+02 $\pm$ 2.375e+00} \\
1000 & 1.397e+03 $\pm$ 1.816e+01 & 1.551e+03 $\pm$ 1.931e+01 & 1.400e+03 $\pm$ 2.402e+01 & 1.462e+03 $\pm$ 1.901e+01 & \textbf{1.537e+02 $\pm$ 2.023e+00} \\
2000 & 1.963e+03 $\pm$ 3.040e+01 & 2.180e+03 $\pm$ 3.540e+01 & 1.961e+03 $\pm$ 1.546e+01 & 2.051e+03 $\pm$ 2.183e+01 & \textbf{1.476e+02 $\pm$ 3.210e+00} \\
5000 & 2.398e+03 $\pm$ 4.369e+01 & 3.476e+03 $\pm$ 6.006e+01 & 3.108e+03 $\pm$ 5.117e+01 & 3.281e+03 $\pm$ 5.445e+01 & \textbf{1.425e+02 $\pm$ 2.512e+00} \\
8000 & 3.237e+03 $\pm$ 4.825e+01 & 4.390e+03 $\pm$ 7.056e+01 & 3.947e+03 $\pm$ 7.011e+01 & 4.133e+03 $\pm$ 4.995e+01 & \textbf{1.415e+02 $\pm$ 2.772e+00} \\
\bottomrule
\end{tabular}
}
}\par
\vspace{2pt}
\parbox{0.96\linewidth}{\footnotesize Note: each entry is reported as mean $\pm$ standard deviation over 10 independent trials. Boldface marks the best mean value in each row.}
\end{sidewaystable}


\subsection{Supervised Hashing for Image Retrieval}
\label{subsec:hashing}

Supervised hashing tests the penalty strategy when solution quality is
measured by retrieval performance.

Following Xiao et al.~\cite{xiao2026relaxation}, let
$Y\in\mathbb R^{n\times m}$ be
the label matrix, $H\in\{-1,1\}^{n\times r}$ the binary code matrix, and
$W\in\mathbb R^{r\times m}$ a projection matrix.  A supervised model is
\begin{equation}
\min_{H\in\{-1,1\}^{n\times r},\,W\in\mathbb R^{r\times m}}
\|rY-HW\|_F^2+\|W\|_F^2
\quad\text{s.t.}\quad H^\top H=nI_r.
\label{eq:hashing}
\end{equation}
In the implementation, we solve a fixed-parameter subproblem of the joint model
\eqref{eq:hashing}. The parameter $W_0$ is obtained by ridge regression from
the initial code matrix $H^0$ and $rY$, and is held fixed during the SPBO solve.
For fixed $W_0$, the objective in \eqref{eq:hashing} reduces, on the
orthogonality-constrained set, to minimizing
$-\langle H^\top(rY),W_0\rangle$.  We use this linear objective, scaled by
$\sqrt n$, in the box-constrained penalty subproblems.  After sign rounding,
we recompute the projection matrix to evaluate retrieval performance.
Unlike the model in \cite{xiao2026relaxation}, we do not impose a separate
bit-balance constraint.
We use MNIST, CIFAR-10, Caltech-256 with 2048-dimensional handcrafted
features, and Caltech-256 with 1024-dimensional CNN features.  Each dataset
uses one fixed random split and initialization, and the code lengths are
$r\in\{16,32,48,64,96\}$.  The nine methods are \alg{SPBO}, \alg{MGBO},
\alg{DPLM}, \alg{ABMO}, \alg{SDH}, \alg{DCSH}, \alg{FSDH}, \alg{LFH}, and
\alg{COSDISH}.  MGBO uses $30$ outer iterations, five subiterations,
function and gradient tolerances of $10^{-5}$
and $10^{-10}$, respectively, and penalty parameter 200.  DPLM and ABMO use
at least 20 iterations with stopping tolerance $10^{-6}$; SDH uses at most
200 iterations with tolerance $10^{-6}$; and LFH uses 50 stochastic-learning
iterations.

\begin{figure}[!tbp]
\centering
\includegraphics[width=0.94\textwidth]{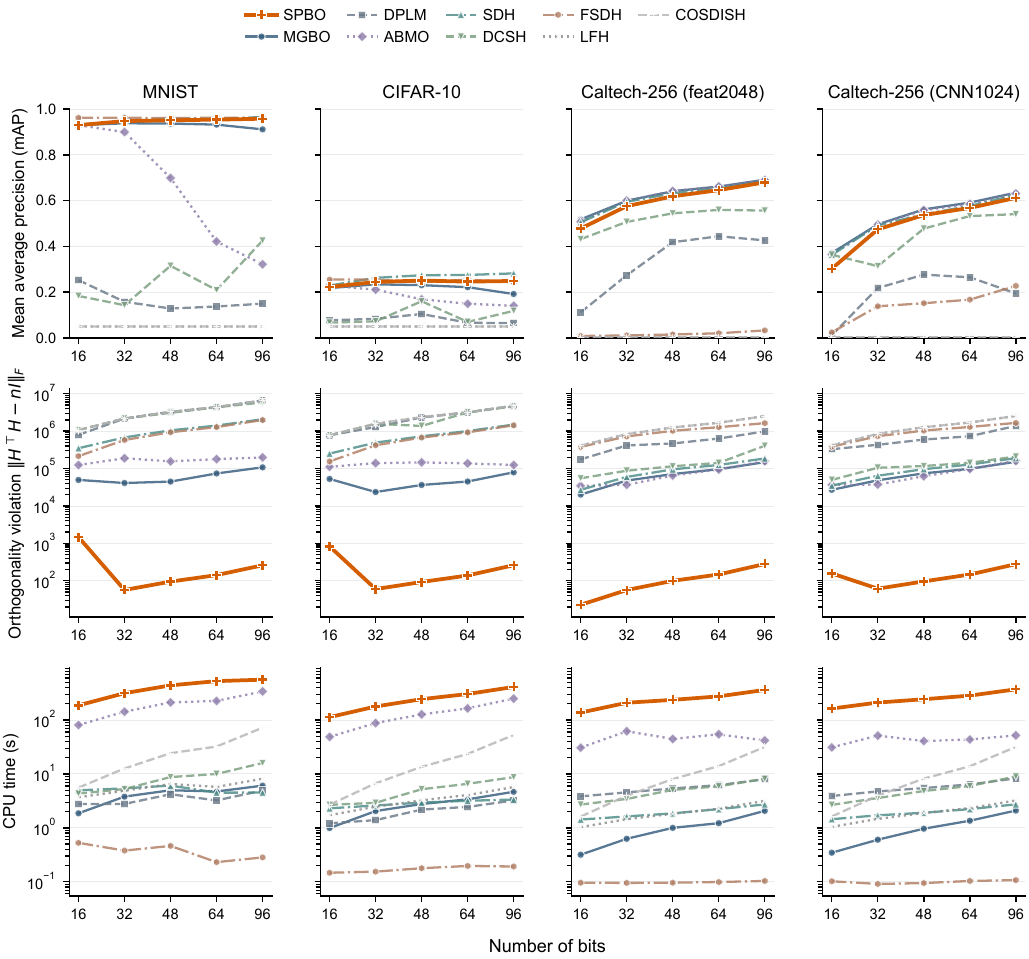}
\caption{Retrieval quality, orthogonality violation, and CPU time across four
datasets and five code lengths.}
\label{fig:hashing_summary}
\end{figure}

Figure~\ref{fig:hashing_summary} reports mean average precision (mAP),
orthogonality violation, and CPU time for all four datasets.  SPBO has the
smallest orthogonality violation in all 20
dataset--code-length combinations.
Compared with the next-smallest violation, it is lower by factors ranging
from $33.7$ to $885$.  Its mAP ranks in the top three in nine combinations,
fourth in ten, and fifth in one.  SPBO has the largest CPU time in all 20
combinations.


\FloatBarrier
\subsection{Hadamard Matrix Completion}
\label{subsec:hadamard}

Finally, Hadamard completion tests orthogonal feasibility when only a subset of
the target entries is observed. Let $\mathcal I$ denote the observed index set,
let $\widetilde X$ contain the observed entries, and let
$\mathcal P_{\mathcal I}$ retain entries indexed by $\mathcal I$ and set all
other entries to zero. We solve
\begin{equation}
\min_{X\in\{-1,+1\}^{n\times n}}
\|\mathcal P_{\mathcal I}(X-\widetilde X)\|_F^2
\quad\text{s.t.}\quad X^\top X=nI_n.
\label{eq:hadamard}
\end{equation}
We use
$n\in\{16,20,24,32,40,48,64,80,96\}$, 150 observation ratios equally
spaced from $0.10$ to $0.90$, and 30 independent runs at every grid point.  The
comparison includes \alg{SPBO}, \alg{MGBO}, \alg{Relax}, \alg{AM}, and
\alg{ABMO}.  Relax, AM, ABMO, and MGBO use a maximum of $1000$ iterations;
their common stopping tolerance is $10^{-4}$, except for Relax's function
tolerance of $10^{-12}$.  MGBO uses $20$ subiterations.
A run is numerically feasible when
$\|X^\top X-nI_n\|_F<10^{-2}$.

Figure~\ref{fig:hadamard_ridgeline} summarizes the mean orthogonality
violation over 30 independent runs.  At the grid point nearest $0.5$,
$\rho_{\rm obs}=0.4973$, SPBO has the lowest mean violation for all five
orders in Table~\ref{tab:hadamard_compare}.  At this observation ratio, SPBO
does not minimize the observed-entry objective or CPU time.
The observed-entry objective and orthogonality violation should be interpreted
together: a small observed-entry error alone does not indicate that the
rounded matrix is orthogonally feasible.
Table~\ref{tab:hadamard_phase_transition} summarizes the empirical success
rates through the first observation ratio at which at least 50\% or 100\% of
the runs are feasible.  SPBO reaches both success levels at the smallest
observation ratio for every reported order.

\begin{figure}[H]
\centering
\includegraphics[width=\textwidth]{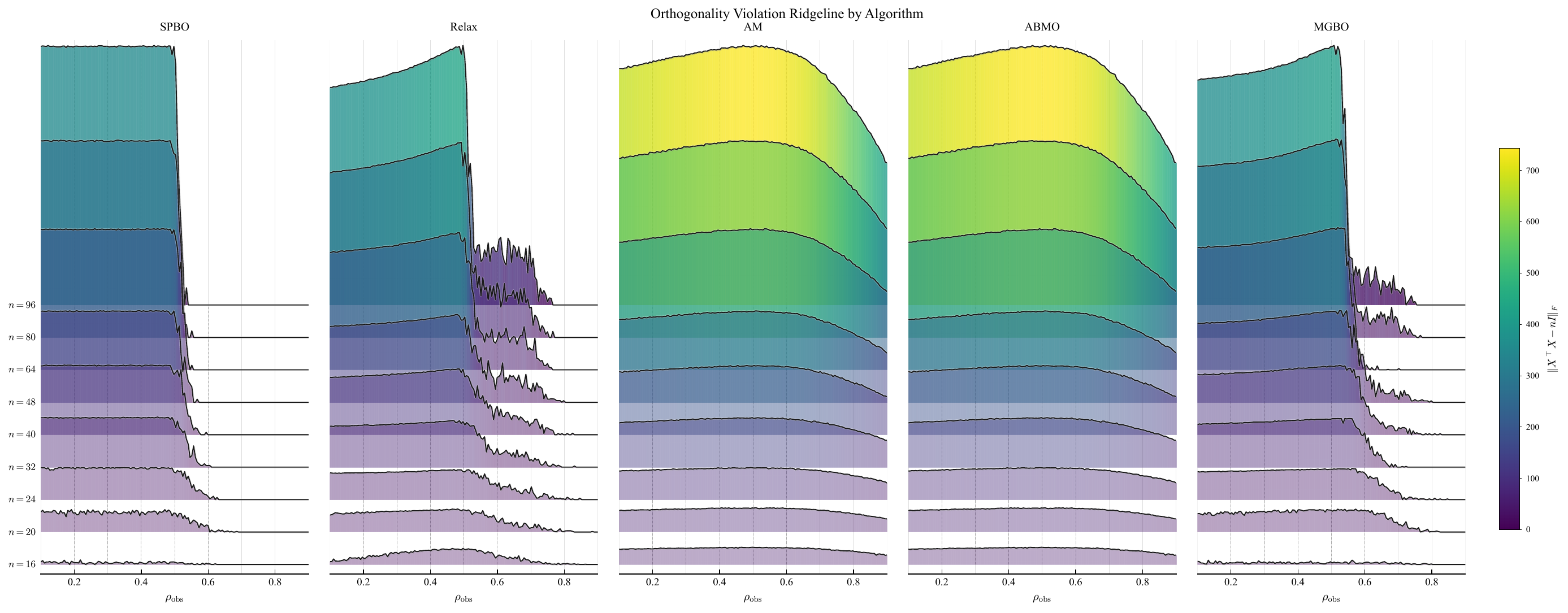}
\caption{Orthogonality violation versus observation ratio for Hadamard
completion over 30 independent runs.}
\label{fig:hadamard_ridgeline}
\end{figure}

\begin{table*}[t]
\centering
\scriptsize
\setlength{\tabcolsep}{3.2pt}
\renewcommand{\arraystretch}{1.06}
\caption{Hadamard observed-entry comparison at $\rho_{\rm obs}=0.4973$. Each entry reports mean $\pm$ sample standard deviation over the available runs; bold marks the best mean.}
\label{tab:hadamard_compare}
\resizebox{\linewidth}{!}{%
\begin{tabular}{@{}lccccc@{}}
\toprule
Method & $n=16$ & $n=32$ & $n=48$ & $n=64$ & $n=96$ \\
\midrule
\multicolumn{6}{l}{\textit{Observed-entry objective}} \\
SPBO & 2.787e+01 $\pm$ 9.247e+00 & 1.437e+02 $\pm$ 5.126e+01 & 3.463e+02 $\pm$ 6.493e+01 & 6.515e+02 $\pm$ 6.726e+01 & 1.628e+03 $\pm$ 1.422e+02 \\
MGBO & 6.507e+01 $\pm$ 1.713e+01 & 5.000e+01 $\pm$ 1.650e+01 & 3.507e+01 $\pm$ 1.412e+01 & 2.173e+01 $\pm$ 8.832e+00 & 4.933e+00 $\pm$ 4.660e+00 \\
Relax & 1.333e-01 $\pm$ 7.303e-01 & \textbf{0.000e+00 $\pm$ 0.000e+00} & \textbf{0.000e+00 $\pm$ 0.000e+00} & \textbf{0.000e+00 $\pm$ 0.000e+00} & \textbf{0.000e+00 $\pm$ 0.000e+00} \\
AM & \textbf{0.000e+00 $\pm$ 0.000e+00} & \textbf{0.000e+00 $\pm$ 0.000e+00} & \textbf{0.000e+00 $\pm$ 0.000e+00} & \textbf{0.000e+00 $\pm$ 0.000e+00} & \textbf{0.000e+00 $\pm$ 0.000e+00} \\
ABMO & \textbf{0.000e+00 $\pm$ 0.000e+00} & \textbf{0.000e+00 $\pm$ 0.000e+00} & \textbf{0.000e+00 $\pm$ 0.000e+00} & \textbf{0.000e+00 $\pm$ 0.000e+00} & \textbf{0.000e+00 $\pm$ 0.000e+00} \\
\midrule
\multicolumn{6}{l}{\textit{Orthogonality violation $\|X^\top X-nI_n\|_F$}} \\
SPBO & \textbf{3.477e+00 $\pm$ 8.097e+00} & \textbf{6.305e+01 $\pm$ 1.722e+01} & \textbf{1.250e+02 $\pm$ 2.285e+00} & \textbf{1.917e+02 $\pm$ 2.377e+00} & \textbf{3.540e+02 $\pm$ 3.011e+00} \\
MGBO & 4.333e+00 $\pm$ 9.181e+00 & 7.701e+01 $\pm$ 2.312e+00 & 1.459e+02 $\pm$ 3.114e+00 & 2.263e+02 $\pm$ 3.528e+00 & 4.145e+02 $\pm$ 4.581e+00 \\
Relax & 2.706e+01 $\pm$ 2.198e+00 & 7.684e+01 $\pm$ 1.466e+01 & 1.388e+02 $\pm$ 3.786e+01 & 2.169e+02 $\pm$ 5.907e+01 & 4.382e+02 $\pm$ 5.332e+00 \\
AM & 4.927e+01 $\pm$ 3.386e+00 & 1.419e+02 $\pm$ 5.552e+00 & 2.592e+02 $\pm$ 4.517e+00 & 4.014e+02 $\pm$ 7.437e+00 & 7.401e+02 $\pm$ 5.924e+00 \\
ABMO & 4.927e+01 $\pm$ 3.386e+00 & 1.419e+02 $\pm$ 5.552e+00 & 2.592e+02 $\pm$ 4.517e+00 & 4.014e+02 $\pm$ 7.437e+00 & 7.401e+02 $\pm$ 5.924e+00 \\
\midrule
\multicolumn{6}{l}{\textit{CPU time (seconds)}} \\
SPBO & 4.033e-03 $\pm$ 4.029e-03 & 5.852e-02 $\pm$ 2.106e-02 & 2.509e-01 $\pm$ 4.903e-02 & 1.227e+00 $\pm$ 3.170e-01 & 2.433e+00 $\pm$ 4.286e-01 \\
MGBO & 5.380e-03 $\pm$ 2.071e-03 & 9.744e-02 $\pm$ 3.300e-02 & 2.721e-01 $\pm$ 6.114e-02 & 5.311e-01 $\pm$ 8.085e-02 & 1.122e+00 $\pm$ 6.290e-02 \\
Relax & 7.075e-03 $\pm$ 2.365e-03 & 5.804e-02 $\pm$ 2.072e-02 & 2.239e-01 $\pm$ 7.181e-02 & 4.727e-01 $\pm$ 1.190e-01 & 1.118e+00 $\pm$ 3.288e-02 \\
AM & 4.945e-04 $\pm$ 1.282e-04 & 2.120e-03 $\pm$ 5.385e-04 & 2.566e-03 $\pm$ 4.840e-04 & 5.849e-03 $\pm$ 1.977e-03 & 1.074e-02 $\pm$ 3.059e-03 \\
ABMO & \textbf{3.251e-04 $\pm$ 1.017e-04} & \textbf{1.640e-03 $\pm$ 5.116e-04} & \textbf{1.913e-03 $\pm$ 3.694e-04} & \textbf{4.974e-03 $\pm$ 1.874e-03} & \textbf{9.235e-03 $\pm$ 2.544e-03} \\
\bottomrule
\end{tabular}%
}
\end{table*}

\begin{table}[htbp]
\centering
\caption{Hadamard feasibility success-rate thresholds over 30 independent runs. Entries are the first observation ratios attaining at least 50\% or 100\% success under $\|X^\top X-nI_n\|_F<10^{-2}$; `--' means the target was not reached.}
\label{tab:hadamard_phase_transition}
\scriptsize
\setlength{\tabcolsep}{4pt}
\begin{tabular}{lcccccccc}
\toprule
\multirow{2}{*}{Method} & \multicolumn{2}{c}{$n=32$} & \multicolumn{2}{c}{$n=48$} & \multicolumn{2}{c}{$n=64$} & \multicolumn{2}{c}{$n=96$} \\
\cmidrule(lr){2-3} \cmidrule(lr){4-5} \cmidrule(lr){6-7} \cmidrule(lr){8-9}
 & $\rho_{50\%}$ & $\rho_{100\%}$ & $\rho_{50\%}$ & $\rho_{100\%}$ & $\rho_{50\%}$ & $\rho_{100\%}$ & $\rho_{50\%}$ & $\rho_{100\%}$ \\
\midrule
SPBO & \textbf{0.540} & \textbf{0.610} & \textbf{0.524} & \textbf{0.556} & \textbf{0.524} & \textbf{0.546} & \textbf{0.513} & \textbf{0.530} \\
MGBO & 0.605 & 0.674 & 0.583 & 0.793 & 0.562 & 0.615 & 0.551 & 0.723 \\
Relax & 0.567 & 0.793 & 0.556 & 0.793 & 0.546 & 0.766 & 0.535 & 0.766 \\
AM & -- & -- & -- & -- & -- & -- & -- & -- \\
ABMO & -- & -- & -- & -- & -- & -- & -- & -- \\
\bottomrule
\end{tabular}
\end{table}

\FloatBarrier

\section{Conclusion}
\label{sec:conclusion}

We studied binary orthogonal optimization through an exact box-constrained
reformulation.  When $\mathcal F_{\rm BO}$ is nonempty, we established local
and global Lipschitz error bounds with an optimal exponent.  The
orthogonality residual bounds the distance to the discrete feasible set and,
below an explicit threshold, certifies exact feasibility after sign rounding.
These bounds also yield a globally exact nonsmooth orthogonality-residual
penalty over the box and a uniformly sharp local penalty threshold.

We developed SPBO, a smoothing penalty method for the box-constrained formulation. 
We quantified the uniform smoothing error,
established first-order complexity for each fixed
smooth subproblem, and proved stationarity of continuation accumulation
points under increasingly accurate subproblem solves.  Since stationarity
alone does not ensure feasibility, we use the residual threshold to certify
feasibility of rounded outputs.  Numerically, SPBO achieved the smallest
orthogonality violation after rounding in all 15 projection and all 20
supervised-hashing configurations, by factors of at least $3.24$ and $33.7$,
respectively, relative to the next-best method.  In Hadamard completion, it
reached the reported numerical-feasibility success levels at lower observation
ratios.  These improvements involve tradeoffs in objective value and CPU time.

\appendix
\section{A One-Sided Linearized Residual Bound}

The following result complements the local error bound by showing that every
nonzero tangent direction at a feasible vertex changes the linearized
orthogonality residual at least linearly.

\begin{proposition}
\label{prop:one_sided_coercivity}
For any $X\in\mathcal F_{\rm BO}$,
\[
 \|DC(X)[H]\|_F\ge2\|H\|_F
 \qquad\forall H\in T_\Omega(X).
\]
\end{proposition}

\begin{proof}
For $H\in T_\Omega(X)$, set $D=-X\circ H\ge0$.  The diagonal entries of
the linearized residual satisfy
\[
 [DC(X)[H]]_{jj}
 =2X_{\cdot j}^\top H_{\cdot j}
 =-2\sum_{i=1}^nD_{ij}.
\]
Consequently,
\[
 \|DC(X)[H]\|_F^2
 \ge4\sum_{j=1}^r\left(\sum_iD_{ij}\right)^2
 \ge4\sum_{i,j}D_{ij}^2
 =4\|H\|_F^2.
\]
\end{proof}

\bibliographystyle{plain}
\bibliography{mybibfile}

\end{document}